\documentclass[12pt,reqno,centertags]{amsart}
\usepackage{a4}
\usepackage{amsmath,amsthm,amscd,amssymb}
\usepackage{latexsym}
\usepackage{upref}
\usepackage{hyperref}
\usepackage{upgreek}
\usepackage{textgreek}
\usepackage{cite}
\usepackage{tikz}
\usepackage{enumerate}
\usepackage{multicol}
\usepackage{accents}
\usepackage{lmodern}
\usepackage{microtype}
\usepackage{graphicx}
\usepackage{subcaption}
\usepackage{dsfont}
\usepackage{colonequals}
\usepackage{bbm}
\usepackage{fullpage}
\UseRawInputEncoding

\newcommand{\abs}[1]{\left|#1\right|}

\newcommand\lrb[1]{\left\lbrace#1 \right\rbrace}
\newcommand\lrp[1]{\left(#1 \right)}

\newcommand\C{{\mathbb C}}
\newcommand\N{{\mathbb N}}
\newcommand\R{{\mathbb R}}

\newcommand\cc[1]{\overline {#1}}

\newcommand\vb[1]{\!\left\langle {#1} \right |}
\newcommand\vk[1]{\left |{#1}\right\rangle\!}
\newcommand\ip[2]{\left\langle {#1},{#2} \right\rangle}
\newcommand\no[1]{\left\| {#1} \right\|}
\newcommand\unit{\hbox{\rm 1\kern-2.8truept l}}
\newcommand\tr[1]{{{\rm tr}}\left(#1\right)}

\DeclareMathOperator{\dom}{dom}
\DeclareMathOperator{\ran}{ran}
\DeclareMathOperator{\Span}{span}

\allowdisplaybreaks
\numberwithin{equation}{section}
 \newtheorem{thm}{Theorem}[section]
 \newtheorem{cor}[thm]{Corollary}
 \newtheorem{lem}[thm]{Lemma}
 \newtheorem{prop}[thm]{Proposition}
 \theoremstyle{definition}
 
 \theoremstyle{remark}
 \newtheorem{rem}[thm]{Remark}
 
  \newtheorem{cex}[thm]{Counterexample}
 \numberwithin{equation}{section}

\begin{document}
\title[Inner products on the Hilbert space $S_2$]{Inner products on the Hilbert space $S_2$\\
of Hilbert--Schmidt operators}

% Information for  author
\author[Josu\'e I. Rios-Cangas]{Josu\'e I. Rios-Cangas}
% Information for  author
\address{Departamento de Matem\'aticas, Universidad Aut\'onoma Metropolitana, Unidad Iztapalapa,  San Rafael Atlixco 186, 093110 Iztapalapa, Ciudad de M\'exico.}
\email{jottsmok@xanum.uam.mx}
\date{\today}

\begin{abstract}
This work presents a rigorous characterization of inner products on the Hilbert space $S_2$ of Hilbert--Schmidt operators. We first deal with a general setting of continuous sesquilinear forms on a Hilbert space $\mathcal H$,  and provide a characterization of all inner products by means of positive operators in $\mathcal {B(H)}$. Next, we establish necessary and sufficient conditions for an operator in $\mathcal B(S_2)$ to be positive. Identifying an inner product with a positive operator enables us to rigorously describe inner products on $S_2$.
\end{abstract}

\maketitle

\section{Introduction}  
 In functional analysis, it is a well-known and useful fact that all norms on a finite-dimensional normed linear space are equivalent. However, this is not the case for infinite-dimensional normed spaces; in fact, one can define uncountably many non-equivalent norms \cite{MR4726393}. Notably, some of the simplest and most practical norms are those induced by inner products. As a matter of fact, all norms on a Banach space that are induced by inner products are equivalent (see Corollary~\ref{cor:equivalence-norms}). This follows from the fact that all inner products on a Hilbert space $\mathcal H$ correspond one-to-one with positive linear operators in $\mathcal{B(H)}$ (see Lemma~\ref{th:univocal-correspondence-IP-PO}).

The class of Hilbert--Schmidt operators $S_2$ belongs to the class of compact operators in $\mathcal{B(H)}$ and plays a key role in quantum mechanics, especially in the formulation of non-commutative quantum mechanics \cite{MR3840084}.
One can endow  $S_2$ with a Hilbert space structure by equipping it with the usual inner product, defined as $(\eta,\tau)\mapsto \tr{\eta^*\tau}$, where $ \eta,\tau\in S_2$. Moreover, it is possible to characterize all continuous sesquilinear forms $\varphi\colon S_2\times S_2\to\C$ by  
\begin{gather}\label{eq:sesq-forms-innerP}
\varphi(\eta,\tau)=\tr{\eta^*A\tau}\,, \qquad A\in\mathcal B(S_2)\,,
\end{gather}
where $\varphi$ is hermitian if and only if $A$ is selfadjoint. Furthermore, $\varphi$ defines either an inner product or a definite inner product on $S_2$ if and only if $A$ is positive or positive definite, respectively (Lemma~\ref{th:cont-sesq-forms}).

The aim of this paper is to present a rigorous characterization of inner products on $S_2$, achieved through a fine-tuned analysis of positive operators in $\mathcal B(S_2)$. This approach was rigorously detailed in \cite{MR3119877} for the space of complex square matrices, and this work seeks to extend those results to the infinite-dimensional case. To achieve this goal, we first exhibit a representation of linear operators and selfadjoint operators in $\mathcal B(S_2)$ (Theorems~\ref{th2:general-decomosition-A} and~\ref{th:selfadjoint-decomposition}). After that, we establish necessary and sufficient conditions under which an operator in $\mathcal B(S_2)$ is positive (Subsections~\ref{ss-necc} and~\ref{ss-ssuff}). These results allow us to rigorously describe inner products on $S_2$. Roughly speaking, any continuous sesquilinear form 
\begin{gather*}
\varphi(\eta,\tau)=\sum_n\tr{\eta^*a_n\tau b_n}\,,\quad \eta,\tau\in S_2\,,
\end{gather*}
defines an inner product on $S_2$, if $a_1,\ldots\geq0$ and $b_1,\ldots>0$ in $\mathcal{B(H)}$, with $\cap_n \ker a_n=\{0\}$ (Theorem~\ref{th:producing-IP}). On the other hand, if $\varphi$ is a definite inner product on $S_2$ then there exist  $a_1,\ldots\geq0$, with $\cap_n \ker a_n=\{0\}$, and positive definite operators $b_1,\ldots$, in $\mathcal{B(H)}$, such that $\varphi$ satisfies either (Theorem~\ref{th:describing-IP})
\begin{align}\label{eq:desc-forms-innerP1}
\varphi(\eta,\tau)&=-\tr{\eta^* a_1\tau b_1}+\sum_{n\geq2}\tr{\eta^* a_n\tau b_n}\quad\mbox{or}\\\label{eq:desc-forms-innerP2}
\varphi(\eta,\tau)&=\sum_{n}\tr{\eta^* a_n\tau b_n}.
\end{align}
The inner product $\varphi$ always holds~\eqref{eq:desc-forms-innerP2} if 
its corresponding operator given in \eqref{eq:sesq-forms-innerP} fulfills $A\eta=\sum_{n=1}^m\hat a_n\tau\hat b_n$, with $\{\hat a_n\}_{n=1}^m,\{\hat b_n\}_{n=1}^m\in\mathcal{B(H)}$, for $m=1,2$ (Theorem~\ref{th:one-two-sumas-IP}). The appearance of the negative sign in the first term of~\eqref{eq:desc-forms-innerP1} is unsettling, but it cannot always be avoided due to Counterexample~\ref{cex:four-sums} (cf. \cite[Sec. 4.2]{MR3119877}).

For the convenience of the reader we address in Section~\ref{s2} the representation of all continuous sesquilinear forms (including inner products) on a Hilbert space $\mathcal H$. This discussion serves as a guideline for Section~\ref{s3}, in which we rigorously characterize continuous sesquilinear forms on $S_2$. In particular, Subsection~\ref{ss-necc} is devoted to presenting necessary conditions, 
while Subsection~\ref{ss-ssuff} addresses sufficiency conditions, both of which determine when an operator in $\mathcal B(S_2)$ is non-negative. The final subsection provides a rigorous characterization of all continuous sesquilinear forms, including inner products, on $S_2$.

%%%%%%%%%%%%%%%
\section{Some aspects about continuous sesquilinear forms on Hilbert spaces}\label{s2}

For a separable Hilbert space $(\mathcal H, \ip{\cdot}{\cdot})$ over $\C$, with inner product anti-linear in its left argument, a linear operator $T\colon \dom T\subset \mathcal H\to \mathcal H$ is \emph{symmetric} if
\begin{gather}\label{eq:non-negative-operator}
\ip{f}{Tf}\in \R\,,\qquad \mbox{for all }\,f\in\dom T\,.
\end{gather}
Besides, $T$ is \emph{non-negative} ($T\geq0$) if \eqref{eq:non-negative-operator} is non-negative, and it is \emph{positive} ($T>0$) if \eqref{eq:non-negative-operator} is positive, for all $f\in\dom T$ non-zero. Moreover,  
a positive operator $T$ is \emph{positive definite} if 
\begin{gather*}
\inf \lrb{\ip{f}{Tf}\,:\,f\in\dom T\,,\no f=1}>0\,.
\end{gather*}
The notions of being positive and positive definite coincide when the Hilbert space satisfies $\dim \mathcal H<+\infty$. 

It is not hard to show that a symmetric operator $T$ with $\dom T=\mathcal H$ is selfadjoint in $\mathcal{B(H)}$ (the class of all bounded operators with domain the whole space $\mathcal H$). Besides, a selfadjoint operator $T \in\mathcal{B(H)}$ satisfies
\begin{gather}\label{eq:glb-T}
m_T\colonequals\inf \lrb{\ip{f}{Tf}\,:\,f\in\dom T\,,\no f=1}\in \R\,,
\end{gather}
which is called the \emph{greatest lower bound} (briefly g.l.b.) of $T$. In this fashion, $T \in\mathcal{B(H)}$ is positive definite if $m_T>0$.

\begin{rem}\label{rm: biyection-SF-LO}
Any continuous sesquilinear form  $\varphi:\mathcal H\times\mathcal H\to \C$ admits the following representations (cf. \cite[Th.\,6,\,Sec.\,2.4]{MR1192782})
\begin{gather}\label{eq:examples-sq-forms}
\varphi(f,g)=\ip{Sf}{g}\,;\quad \varphi(f,g)=\ip{f}{Tg}\,,\qquad T,S\in\mathcal{B(H)}\,,
\end{gather}
with $\no\varphi=\no T=\no S$. Thereby, one can define a continuous sesquilinear form $\varphi$ \emph{hermitian, non-negative, positive or positive definite}, if the operator $T$ in \eqref{eq:examples-sq-forms} is selfadjoint, non-negative, positive or positive definite, respectively.
\end{rem}

It readily follows that a sesquilinear form is continuous and positive if and only if it defines an inner product on $\mathcal H$. So, we may say that an inner product is \emph{definite}, if it is a continuous positive definite sesquilinear form.

The following assertion is a simple matter from Remark~\ref{rm: biyection-SF-LO}, which exhibits a characterization of all inner products on $\mathcal H$.

\begin{lem}\label{th:univocal-correspondence-IP-PO}
All the inner products on a Hilbert space $(\mathcal H,\ip{\cdot}{\cdot})$ are in one-to-one correspondence with the positive linear operators in $\mathcal{B(H)}$. Every inner product $\varphi$ on $\mathcal H$ admits the representation 
\begin{gather}\label{eq:inner-P-rep}
\varphi(\cdot,\cdot)=\ip{\cdot}{T\cdot}\,,\quad 0<T\in \mathcal{B(H)}
\end{gather}
and satisfies $\no{\varphi}=\no T$. Moreover, \eqref{eq:inner-P-rep} is a definite inner product if and only if $T$ is positive definite.
\end{lem}

In the following, we  provide a special property of norms induced by inner products, on Banach spaces.

\begin{cor}\label{cor:equivalence-norms}
In a Banach space $\mathcal H$,  all norms induced by inner products are equivalent.
\end{cor}
\begin{proof}
Let $\varphi_1,\varphi_2$ be inner products on $\mathcal H$ and 
\begin{gather}\label{eq:norms-equivalence-1}
\no f_{\varphi_j}={\varphi_j}(f,f)^{1/2}\,,\qquad f\in\mathcal H\,,\quad  j=1,2\,.
\end{gather}
By Lemma~\ref{th:univocal-correspondence-IP-PO}, there exist unique positive operators $T_1,T_2\in\mathcal{B(H)}$, such that $\varphi_1(\cdot,\cdot)=\varphi_2(\cdot,T_1\cdot)$ and  $\varphi_2(\cdot,\cdot)=\varphi_1(\cdot,T_2\cdot)$. So, by Cauchy–Schwarz inequality  and \eqref{eq:norms-equivalence-1}, one simply computes that $\no f_{\varphi_1}^2\leq\no f_{\varphi_2}^2\no{T_1}_{\varphi_2}$ and  $\no f_{\varphi_2}\leq\no f_{\varphi_1}^2\no{T_2}_{\varphi_1}$. Therefore,
\begin{gather*}
\no{T_1}_{\varphi_2}^{-1/2}\no f_{\varphi_1}\leq \no f_{\varphi_2}\leq\no f_{\varphi_1}^2\no{T_2}_{\varphi_1}^{1/2}\,,\quad f\in\mathcal H\,,
\end{gather*}
as required.
\end{proof}

In what follows, we shall use the univocal correspondence stablished in Lemma~\ref{th:univocal-correspondence-IP-PO} to rigorously characterize inner products on the class of Hilbert--Schmidt operators.

\section{Continuous sesquilinear forms and inner products on $S_2$}\label{s3}

We regard \emph{the class of Hilbert--Schmidt operators} $S_2\subset \mathcal {B(H)}$, which consists of those operators $\eta\colon \mathcal H\to\mathcal H$ such that
\begin{gather}\label{eq:norm-HS}
\no\eta_2\colonequals\lrp{\sum_n\no{\eta u_n}^2}^{1/2}<+\infty\,, 
\end{gather}
being $\lrb{u_n}$ an orthonormal basis in $\mathcal H$, where the sum in \eqref{eq:norm-HS} is independent of the choice of basis. %Besides, the sequence $\{\lambda_n\}$ of the singular eigenvalues of $\eta$ (the eigenvalues of $\abs{\eta}$) belongs to $\ell_2(\N)$ and satisfies $\no \eta_2^2=\sum_ns_n^2$. Moreover, $(S_2,\no\cdot_2)$ is a separable Banach space (cf. \cite[Sect.\,11.3]{MR1192782}).

It is well-know that the space $S_2$ is naturally endowed with a Hilbert structure; viz. $(S_2,\ip{\cdot}{\cdot}_2)$ is a separable Hilbert space. The usual inner product on $S_2$ is given by 
\begin{gather}\label{eq:inner-product-S2}
\ip{\eta}{\tau}_2=\tr{\eta^*\tau}\,,\qquad \eta_\tau\in S_2\,,
\end{gather}
which induces the norm $\no\cdot_2$ given in \eqref{eq:norm-HS}. 

The expression \eqref{eq:inner-product-S2} allows characterizing continuous sesquilinear forms on $S_2$, by means of operators in $\mathcal B(S_2)$.

\begin{lem}\label{th:cont-sesq-forms}
Every continuous sesquilinear  form $\varphi\colon S_2\times S_2\to\C$ admits the representation 
\begin{gather*}
\varphi(\eta,\tau)=\tr{\eta^*A\tau}\,, \qquad \eta,\tau\in S_2\,,
\end{gather*}
where $A\in\mathcal B(S_2)$ and satisfies  $\no A=\no{\varphi}$. Besides, 
\begin{enumerate}[(i)]
\item\label{it1:inn-p} $\varphi$ is hermitian if and only $A$ is selfadjoint, and
\item\label{it2:inn-p} $\varphi$ is (definite) inner product on $S_2$ if and only if $A$ is positive (definite).
\end{enumerate}
\end{lem}
\begin{proof}
It is straightforward from Remark~\ref{rm: biyection-SF-LO}, Lemma~\ref{th:univocal-correspondence-IP-PO} and \eqref{eq:inner-product-S2}. 
\end{proof}

By virtue of Lemma~\ref{th:cont-sesq-forms}, we are now motivated to describe linear operators in $\mathcal B(S_2)$. This analysis provides a rigorous characterization of continuous sesquilinear forms, particularly inner products, on $S_2$.

\subsection{Non-negative selfadjoint operators in $\mathcal B(S_2)$ (necessity)}\label{ss-necc}
Consider the system of rank-one operators 
\begin{gather*}
\varepsilon_{nm}\colonequals\vk{e_n}\vb{e_m}\in S_2\subset \mathcal {B(H)}\,,\quad n,m\in\N
\end{gather*}
as a basis for $\mathcal{B(H)}$ and as the canonical basis for $S_2$, where $\{e_n\}$ denotes the canonical basis for $\mathcal H$. Thus, one has for $f\in\mathcal H$ that $\varepsilon_{nm}f=\ip{e_m}{f}e_n$  and
\begin{gather*}
\varepsilon_{nm}\varepsilon_{jk}= \delta_{mj}\varepsilon_{nk}\,;\quad \ip{\varepsilon_{nm}}{\varepsilon_{jk}}_2=\delta_{nj}\delta_{mk}\,,\qquad n,m,j,k\in\N\,,
\end{gather*}
being $\delta_{mn}$ the \emph{Kronecker delta function}.

For $a,b\in \mathcal{B(H)}$, and $\eta\in S_2$, it follows that $b\eta a\in S_2$, since one readily checks that 
\begin{gather*}
\no{b\eta a}_2\leq\no a\no b\no \eta_2\,,
\end{gather*}
viz. $S_2$ is an ideal on $ \mathcal{B(H)}$. 

For subsets $\{a_n=\sum_{hk}\alpha_{hk}^{(n)}\varepsilon_{hk}\},\{b_n=\sum_{rm}\beta_{rm}^{(n)}\varepsilon_{rm}\}\subset \mathcal{B(H)}$, with $\alpha_{hk}^{(n)},\beta_{rm}^{(n)}\in\C$, the linear application  
 \begin{gather*}
 A\eta=\sum_na_n\eta b_n\,,\qquad \eta\in\Span\{\varepsilon_{nm}\}\subset S_2\,,
 \end{gather*}
 is a densely defined linear operator $A\colon S_2\to S_2$ if and only if 
 \begin{gather}\label{eq:condition-anbn}
\sum_{h,s}\abs{\sum_n\alpha_{hj}^{(n)}\beta_{l,n}^{(n)}}^2<+\infty\,,\qquad \mbox{for all }\, j,l\in\N\,.
\end{gather}
The above is true since for $j,l\in\N$,
\begin{align*}
\no{A\varepsilon_{jl}}_2^2&=\sum_s\no{\sum_na_n\varepsilon_{jl}b_ne_s}^2\\&=\sum_s\no{\sum_{n,h}\alpha_{hj}^{(n)}\beta_{l,s}^{(n)}e_h}^2=\sum_{h,s}\abs{\sum_n\alpha_{hj}^{(n)}\beta_{l,s}^{(n)}}^2\,.
\end{align*}

\begin{thm}\label{th:A-nonnegative}
If $\{a_n\},\{b_n\}\subset \mathcal{B(H)}$ satisfy \eqref{eq:condition-anbn} then the linear operator
 \begin{gather*}
 A\eta=\sum_na_n\eta b_n\,,\qquad \eta\in\Span\{\varepsilon_{nm}\}\subset S_2\,, \quad \mbox{ is:}
 \end{gather*}
\begin{enumerate}[(i)]
\item\label{it1:sA} Symmetric if $a_1,\ldots,b_1,\ldots$, are selfadjoint.
\item\label{it2:sA} Non-negative if $a_1,\ldots,b_1,\ldots\geq0$.
\item\label{it3:sA} Positive if $a_1,\ldots>0$ and $b_1,\ldots\geq0$. with $\cap_n \ker b_n=\{0\}$.
\item\label{it4:sA} Positive definite if $a_1,\ldots$ are positive definite, with g.l.b. $m_{a_n}$ (see \eqref{eq:glb-T}), and 
$b_1,\ldots\geq0$, such that 
\begin{gather*}
m_a=\inf_n\{m_{a_n}\}>0\quad\mbox{and}\quad\cap_n \ker b_n=\{0\}\,.
\end{gather*}
\end{enumerate}
The same result holds in items~\eqref{it3:sA}-\eqref{it4:sA} if the conditions of $\{a_n\}$ and $\{b_n\}$ are interchanged.

\end{thm}
\begin{proof}
If $a_1,\ldots,b_1,\ldots$, are selfadjoint, then it follows for $\eta\in\dom A$ that
\begin{gather*}
\ip{\eta}{A\eta}_2=\sum_n\tr{\eta^*a_n\eta b_n}=\sum_n\tr{(a_n\eta b_n)^*\eta}=\ip{A\eta}{\eta}_2\,,
\end{gather*}
which implies \eqref{it1:sA}. For $a_1,\ldots,b_1,\ldots\geq0$, one has
\begin{gather}\label{eq:nonneg-A}
\ip{\eta}{A\eta}_2=\sum_n\ip{\eta}{a_n^{1/2}a_n^{1/2}\eta b_n^{1/2}b_n^{1/2}}_2=\sum_n\no{a_n^{1/2}\eta b_n^{1/2}}_2^2\geq0\,,
\end{gather}
which implies \eqref{it2:sA}. Now for \eqref{it3:sA}, if we additionally suppose $a_1,\ldots>0$ and $\cap_n \ker b_n=\{0\}$, then from \cite[Lem.\,3.2.1.2]{MR1192782} and \cite[Prop.\,1.6]{MR2953553}, 
\begin{gather}\label{eq:span-bn-H}
\cc{\displaystyle\Span_n\lrb{\cc{\ran b_n}}}=\lrp{\cap_n\lrp{\cc{\ran b_n}}^\perp}^\perp=\lrp{\cap_n\ker b_n}^\perp=\mathcal H\,.
\end{gather}
Thus, for non-zero $\eta\in\dom A$, if $\ip{\eta}{A\eta}=0$ then one has by \eqref{eq:nonneg-A} that $\eta b_n=0$, since $a_n$ is invertible. So, $\eta\ran b_n=\{0\}$, for all $n\in\N$, and by \eqref{eq:span-bn-H},
\begin{gather*}
\{0\}=\eta\lrp{\cc{\displaystyle\Span_n\lrb{\cc{\ran b_n}}}}=\eta(\mathcal H)\,,
\end{gather*}
i.e., $\eta=0$, a contradiction. Hence, $\ip{\eta}{A\eta}>0$. To prove \eqref{it4:sA}, for $\eta\in\dom A$ non-zero, it follows by item \eqref{it3:sA} that there exits $j\in\N$ such that $\eta b_j\neq 0$. In this fashion, taking into account \eqref{eq:norm-HS} and \eqref{eq:nonneg-A}, one computes that 
\begin{align*}
\ip{\eta}{A\eta}_2&\geq\no{a_j^{1/2}\eta b_j^{1/2}}_2^2=\sum_n\ip{\eta b_j^{1/2}u_n}{a_j\eta b_j^{1/2}u_n}\geq m_{a_j}\geq m_a\,,
\end{align*}
whence it follows that $A$ is positive definite. 

The proof when $a_1,\ldots\geq 0$, with $\cap_n \ker a_n=\{0\}$ and $b_1,\ldots$ are positive (resp. positive definite) in item \eqref{it3:sA} (resp. \eqref{it4:sA}), follows the same above reasoning and taking adjoints. 
\end{proof}

We emphasize that Theorem~\ref{th:A-nonnegative}.\eqref{it3:sA} cannot be relaxed to 
\begin{gather*}
a_1,\ldots,b_1,\ldots\geq0\,,\quad \mbox{with}\quad \cap_n \ker a_n=\cap_n \ker b_n=\{0\}\,.
\end{gather*}
Indeed, for $n\in\N$, $\varepsilon_{nn}\geq0$ and $\cap_n \ker \varepsilon_{nn}=\{0\}$. However, for $h,k\in\N$,
\begin{gather*}
\ip{\varepsilon_{hk}}{\sum_{n}\varepsilon_{nn}\varepsilon_{hk}\varepsilon_{nn}}_2=\ip{\varepsilon_{hk}}{\varepsilon_{hk}\varepsilon_{hk}}_2=0\,,\qquad j\neq k\,.
\end{gather*}

\begin{rem}\label{rm:selfadjoint-Asuff} The operator $A$ in Theorem~\ref{th:A-nonnegative}.\eqref{it1:sA}-\eqref{it4:sA} is as well selfadjoint and bounded, if $\dom A=S_2$.
\end{rem}

\subsection{Non-negative selfadjoint operators in $\mathcal B(S_2)$ (sufficiency)}\label{ss-ssuff}

We shall first characterize the operators in $\mathcal B(S_2)$. Recall that the operators $A$ and $A^*$ belong to $\mathcal B(S_2)$ simultaneously.

\begin{lem}\label{lem:1std-ecomposition-A}
Every linear operator $A\in\mathcal B(S_2)$ is decomposed into
\begin{gather}\label{eq:A-decomposition}
A\eta=\sum_{n,m}\varepsilon_{nm}\eta a_{nm}\,,\qquad \eta\in S_2\,,
\end{gather}
where $a_{nm}\colonequals \sum_{j,k}\ip{\varepsilon_{nk}}{A\varepsilon_{mj}}_2\varepsilon_{jk}\in \mathcal{B(H)}$.
\end{lem}
\begin{proof} We first show that $a_{nm}\in \mathcal{B(H)}$. Since  $A\in\mathcal B(S_2)$, it follows that  $A^*\in\mathcal B(S_2)$ and $A^*\varepsilon_{nk}=\sum_{j,s}\ip{\varepsilon_{sj}}{A^*\varepsilon_{nk}}_2\varepsilon_{sj}$. Thus, for $f=\sum_tf_te_t\in\mathcal H$, with $f_t\in\C$, one has that $\sum \cc f_k\varepsilon_{nk}\in S_2$, by virtue of $\no f=\no{\sum \cc f_k\varepsilon_{nk}}_2$. Thereby, $A^*(\sum \cc f_k\varepsilon_{nk})\in S_2\subset \mathcal{B(H)}$ and $(A^*(\sum \cc f_k\varepsilon_{nk}))^*\in \mathcal{B(H)}$. Thence, \begin{align*}
a_{nm}f&=\sum_{j,k,t}f_t\ip{\varepsilon_{nk}}{A\varepsilon_{mj}}\varepsilon_{jk}e_t=\sum_{j,k}f_k\ip{\varepsilon_{nk}}{A\varepsilon_{mj}}e_j\\
&=\sum_{j,s,k}f_k\ip{A^*\varepsilon_{nk}}{\varepsilon_{sj}}\varepsilon_{js}e_m=\lrp{A^*\lrp{\sum \cc f_k\varepsilon_{nk}}}^*e_m\in\mathcal H\,,
\end{align*}
whence it follows that $\dom a_{nm}=\mathcal H$. Now, $A\varepsilon_{ls}=\sum_{n,k}\ip{\varepsilon_{nk}}{A\varepsilon_{ls}}_2\varepsilon_{nk}$ and  
\begin{gather}\label{eq:bounded-anm}
\varepsilon_{rn}A\varepsilon_{ms}=\varepsilon_{rs}a_{nm}\,,\quad r,s,n,m\in\N\,.
\end{gather}
Since $\varepsilon_{rn}A\varepsilon_{ms}, \varepsilon_{rs}\in S_2\subset \mathcal{B(H)}$, for a sequence $\{f_t\}\subset\mathcal H$ such that $f_t\to f$ and $a_{nm}f_t\to g$, it follows by \eqref{eq:bounded-anm} that
\begin{align*}
\varepsilon_{rr}a_{nm}f&=\varepsilon_{rn}A\varepsilon_{mr}f=\lim_t\varepsilon_{rn}A\varepsilon_{mr}f_t\\
&=\lim_t\varepsilon_{rr}a_{nm}f_t=\varepsilon_{rr}g.
\end{align*}
Then, $a_{nm}f=\sum_{r}\varepsilon_{rr}a_{nm}f=\sum_{r}\varepsilon_{rr}g=g$. Hence, $a_{nm}$ is a closed operator defined in the whole space $\mathcal H$, viz. $a_{nm}\in\mathcal{B(H)}$. 

Now, if $\eta=\sum_{l,s}\eta_{ls}\varepsilon_{ls}\in S_2$, with $\eta_{ls}\in\C$, then
\begin{align*}
A\eta&=\sum_{l,s,n,k}\eta_{ls}\ip{\varepsilon_{nk}}{A\varepsilon_{ls}}_2\varepsilon_{nk}=\sum_{n,m,l,s,j,k}\varepsilon_{nm}\eta_{ls}\varepsilon_{ls}\ip{\varepsilon_{nk}}{A\varepsilon_{mj}}_2\varepsilon_{jk}\\
&=\sum_{n,m,l,s}\varepsilon_{nm}\eta_{ls}\varepsilon_{ls}a_{nm}=\sum_{n,m}\varepsilon_{nm}\eta a_{nm}\,,\end{align*}
which implies \eqref{eq:A-decomposition}. 
\end{proof}
\begin{rem}
The decomposition \eqref{eq:A-decomposition} is not unique. Indeed, one can also decompose any linear operator $A\in\mathcal B(S_2)$ into
\begin{gather*}
A\eta=\sum_{n,m}\hat a_{nm}\eta \varepsilon_{nm}\,,\qquad \eta\in S_2\,,
\end{gather*}
with $\hat a_{nm}=\sum_{j,k}\ip{\varepsilon_{jm}}{A\varepsilon_{kn}}_2\varepsilon_{jk}\in\mathcal{B(H)}$. 
\end{rem}

The proof of the following assertion is straightforward from Lemma~\ref{lem:1std-ecomposition-A}.

\begin{thm}\label{th2:general-decomosition-A} If $A\in\mathcal B(S_2)$ then there exist $\{a_n\},\{b_n\}\subset\mathcal{B(H)}$, such that
\begin{gather}\label{eq:general-decompositionA}
A\eta=\sum_{n}a_n\eta b_{n}\,,\qquad \eta\in S_2\,.
\end{gather}
\end{thm}

\begin{rem}\label{rm:sum-LI}
If \eqref{eq:general-decompositionA} is such that 
\begin{align}\label{eq:sum-LI}
A\eta=\sum_{n=1}^ma_n\eta b_{n}\,,
\end{align}
then $\{a_n\}_{n=1}^m$ and $\{b_n\}_{n=1}^m$ can be assumed linearly independent (briefly l.i.).
Indeed, if $\{a_n\}_{n=1}^m$ is not l.i., then it contains a l.i. subset, namely $\{a_n\}_{n=1}^r$, for which $a_j=\sum_{n=1}^r\alpha_{jn}a_n$, for $j=r+1,\ldots,m$, and
\begin{gather*}
A\eta=\sum_{n=1}^ra_n\eta\lrp{b_{n}+\sum_{j=r+1}^m\alpha_{jn}b_j}\,.
\end{gather*}
Thus, we may suppose that $\{a_n\}_{n=1}^m$ in \eqref{eq:sum-LI} is linearly independent and by above, there exists a l.i. subset of $\{b_n\}_{n=1}^m$, namely $\{b_n\}_{n=1}^r$, such that $b_j=\sum_{n=1}^r\beta_{jn}a_n$, for $j=r+1,\ldots,m$, and
\begin{gather*}
A\eta=\sum_{n=1}^r\lrp{a_{n}+\sum_{j=r+1}^m\beta_{jn}a_j}\eta b_n\,,
\end{gather*}
where it is simple to check that $\{a_{n}+\sum_{j=r+1}^m\beta_{jn}a_j\}_{n=1}^r$ is l.i.
\end{rem}

In what follows, we shall describe selfadjoint operators in $\mathcal B(S_2)$. To do so, we first describe the adjoint of an operator in  $\mathcal B(S_2)$.

\begin{lem}\label{lem:adjoint-A}
For an operator $A$ in $\mathcal B(S_2)$, such that $A\eta=\sum_na_n\eta b_n$, with $\{a_n\},\{b_n\}\subset \mathcal{B(H)}$, it follows that $\sum_na_n^*\eta b_n^*\in S_2$ and
\begin{gather}\label{eq:adjoint-A}
A^*\eta=\sum_na_n^*\eta b_n^*\,,\qquad \eta\in S_2\,.
\end{gather}
\end{lem}
\begin{proof}
Since  $A\in\mathcal B(S_2)$, one has that  $A^*\in\mathcal B(S_2)$. Besides, for $s\in\N$ and $\eta,\rho\in S_2$, it follows  that
$\tau_s=\sum_{n=1}^sa_n^*\eta b_n^*\in S_2$ and 
\begin{gather}\label{eq:wc-Aadjoint}
\lim_s\ip{\rho}{\tau_s}_2=\lim_s\ip{\sum_{n=1}^sa_n\rho b_n}{\eta}_2=\ip{A\rho}{\eta}_2=\ip{\rho}{A^*\eta}_2\,.
\end{gather}
So, $\{\tau_s\}$ converges weakly to $A^*\eta$, viz. $\{\tau_s\}\subset S_2$ is a bounded sequence and thereby $\sum_n a_n^*\eta b_n^*\in S_2$. Therefore, \eqref{eq:wc-Aadjoint} yields $\ip{\rho}{\sum_n a_n^*\eta b_n^*}_2=\ip{\rho}{A^*\eta}_2$, for all $\rho\in S_2$, which implies \eqref{eq:adjoint-A}.
\end{proof}

It follows by Remark~\ref{rm:selfadjoint-Asuff} that if $A\in \mathcal B(S_2)$, for which $A\eta=\sum_na_n\eta b_n$, with $a_1,\ldots,b_1,\ldots \mathcal{B(H)}$ selfadjoint, then $A$ is selfadjoint. The following asserts the converse statement.  

\begin{thm}\label{th:selfadjoint-decomposition}
If $A$ is a selfadjoint operator in $\mathcal B(S_2)$ then there exist selfadjoint operators $a_1,\ldots,b_1,\ldots\in\mathcal{B(H)}$, such that
\begin{gather}\label{eq:selfadjoint-decomposition}
A\eta=\sum_{n}a_n\eta b_{n}\,,\qquad \eta\in S_2\,.
\end{gather}
\end{thm}
\begin{proof}
It follows by Lemma~\ref{lem:1std-ecomposition-A} that $A\eta=\sum_{n,m}\varepsilon_{nm}\eta a_{nm}$, for $\eta\in S_2$, with $a_{nm}=\sum_{j,k}\ip{\varepsilon_{nk}}{A\varepsilon_{mj}}_2\varepsilon_{jk}\in\mathcal{B(H)}$. Besides, since $A$ is selfadjoint, 
\begin{gather}\label{eq:adjoint-anm}
a_{nm}^*=\sum_{j,k}\ip{A\varepsilon_{mj}}{\varepsilon_{nk}}_2\varepsilon_{jk}^*=\sum_{j,k}\ip{\varepsilon_{mj}}{A\varepsilon_{nk}}_2\varepsilon_{kj}=a_{mn}\,.
\end{gather}
In this fashion, Lemma~\ref{lem:adjoint-A} and \eqref{eq:adjoint-anm} imply

\begin{align}\label{eq:selfadjoint-decompo-byAmBm}
\begin{split}
A\eta&=\frac12\lrp{A\eta+A^*\eta}-\frac1{2i}\lrp{\sum_{n,m}\varepsilon_{mn}\eta a_{nm}-\sum_{n,m}\varepsilon_{mn}\eta a_{nm}}\\
&=\frac12\sum_{n,m}(\varepsilon_{nm}\eta a_{nm}+\varepsilon_{mn}\eta a_{mn})-\frac1{2i}\sum_{n,m}(\varepsilon_{nm}\eta a_{mn}-\varepsilon_{mn}\eta a_{nm})\\
&=\sum_{n,m}\lrp{\frac{1+i}{2}\varepsilon_{nm}+\frac{1-i}{2}\varepsilon_{mn}}\eta \lrp{\frac{1-i}{2}a_{nm}+\frac{1+i}{2}a_{mn}}\,,
\end{split}
\end{align}
where clearly $\frac{1+i}{2}\varepsilon_{nm}+\frac{1-i}{2}\varepsilon_{mn}$ and $\frac{1-i}{2}a_{nm}+\frac{1+i}{2}a_{mn}$ are selfadjoint in $\mathcal{B(H)}$. By rearranging and renaming \eqref{eq:selfadjoint-decompo-byAmBm}, one arrives at \eqref{eq:selfadjoint-decomposition}.
\end{proof}

\begin{rem} One can follow the same reasoning of Remark~\ref{rm:sum-LI} to show that if \eqref{eq:selfadjoint-decomposition} is such that 
\begin{align*}
A\eta=\sum_{n=1}^ma_n\eta b_{n}\,,
\end{align*}
then the sets of selfadjoint operators $\{a_n\}_{n=1}^m$ and $\{b_n\}_{n=1}^m$ can be assumed l.i. This, bearing in mind that $\sum_{n=1}^m\alpha_{n}a_n$, with $\alpha_{n}\in\C$, is selfadjoint if and only if $\alpha_{n}\in\R$, for $n=1,\ldots,m$.
\end{rem}

The set of selfadjoint operators $\{\hat\varepsilon_{nm}\}$  where
\begin{gather*}
\hat\varepsilon_{nm}\colonequals\frac{1+i}{2}\varepsilon_{nm}+\frac{1-i}{2}\varepsilon_{mn}\in S_2\,,
\end{gather*}
is an orthonormal basis for $S_2$. Indeed, one computes for $h,k,n,m\in\N$ that 
\begin{align*}
\ip{\hat\varepsilon_{nm}}{\hat\varepsilon_{hk}}_2&=\frac12\ip{\varepsilon_{nm}}{\varepsilon_{hk}}_2-\frac i2\ip{\varepsilon_{nm}}{\varepsilon_{kh}}_2+\frac i2\ip{\varepsilon_{mn}}{\varepsilon_{hk}}_2+\frac12\ip{\varepsilon_{mn}}{\varepsilon_{kh}}_2\\
&=\frac12\delta_{nh}\delta_{mk}-\frac i2\delta_{nk}\delta_{mh}+\frac i2\delta_{mh}\delta_{nk}+\frac12\delta_{mk}\delta_{nh}=\delta_{nh}\delta_{mk}\,.
\end{align*}
Besides, this orthonormal basis satisfies for $n,m\in\N$ that 
\begin{align}\label{it1-saB}
\begin{split}
\hat\varepsilon_{nn}&=\varepsilon_{nn}\\
\ran\varepsilon_{nn}, \ran\varepsilon_{mm}&\subset \ran \hat\varepsilon_{nm}\\
\cap_n \ker \hat\varepsilon_{nm}&=\cap_m \ker \hat\varepsilon_{nm}=\{0\}
\end{split}
\end{align}

In what follows, we address the case when $A\in\mathcal B(S_2)$ is non-negative. In this instance, $A$ satisfies Theorem~\ref{th:selfadjoint-decomposition}, and we wonder if the operators $a_1,\ldots,b_1,\ldots$ in the decomposition \eqref{eq:selfadjoint-decomposition} can be also considered non-negative.

\begin{thm}\label{th:one-sum-Ap} Let $A$ be a  non-zero operator in $\mathcal B(S_2)$ given by $A\eta=a\eta b$, with $a,b\in\mathcal{B(H)}$. If $A\geq0$ then there exist $\hat a,\hat b\geq0$ in $\mathcal{B(H)}$, such that 
\begin{gather}\label{eq:AnN-oneS}
A\eta=\hat a\eta\hat b\,,\quad \eta\in S_2\,.
\end{gather}
Besides, 
$\hat a,\hat b$ are positive or positive definite if $A$ is positive or positive definite, respectively.
\end{thm}
\begin{proof}
Since $A\neq0$, one has that $a,b\neq 0$. Besides,
\begin{gather}\label{eq:ip-AnonN-onesum}
0\leq\ip{\vk{g}\vb {f}}{A\vk{g}\vb {f}}_2=\ip{g}{a g}\ip{f}{bf}\,,\quad f,g\in\mathcal H\,.
\end{gather}
Let $f_0\in\mathcal H$ such that $\ip{f_0}{bf_0}=\alpha\neq 0$. Thus, one obtains by \eqref{eq:ip-AnonN-onesum} that $\ip{g}{\alpha ag}\geq0$, i.e., $\hat a=\alpha a\geq 0$ and \eqref{eq:AnN-oneS}, where $\hat b=b/\alpha\geq0$, since for $g_0\in\mathcal H$ such that $\ip{g_0}{\hat ag_0}=\beta> 0$, one has on account \eqref{eq:ip-AnonN-onesum} that $\beta\langle{f},{\hat bf}\rangle\geq0$. It is straightforward to check from the equality of \eqref{eq:ip-AnonN-onesum} that $\hat a,\hat b$ are positive or positive definite, if $A$ is positive or positive definite, respectively.
\end{proof}
We obtain an analogous result of the last theorem for the case where $A$ is positive definite and consists of exactly two summands.

\begin{thm}\label{th:two-sum-Ap} Let $A\in\mathcal B(S_2)$ given by 
\begin{gather}\label{eq:decm-two-sums-A}
A\eta=a_1\eta b_1+a_2\eta b_2\,, \mbox{with } a_1,a_2,b_1,b_2\in\mathcal{B(H)}\,.
\end{gather}
If $A$ is positive definite then there exist $\hat a_1,\hat a_2\in\mathcal{B(H)}$ non-negative, whit $\ker a_1\cap\ker a_2=\{0\}$, and $b_1,b_2\in \mathcal{B(H)}$ positive definite, such that 
\begin{gather}\label{eq:AnN-twoS}
A\eta=\hat a_1\eta \hat b_1+\hat a_2\eta \hat b_2\,,\quad \eta\in S_2\,.
\end{gather}
\end{thm}
\begin{proof}
We have $a_1,a_2,b_1,b_2\neq 0$, otherwise one can use Theorem~\ref{th:one-sum-Ap}. In this fashion, for $f,g\in\mathcal H$,
\begin{gather}\label{eq:ip-AnonN-twosum}
0<m_A\leq\ip{\vk{g}\vb {f}}{A\vk{g}\vb {f}}_2=\ip{g}{a_1 g}\ip{f}{b_1f}+\ip{g}{a_2 g}\ip{f}{b_2f}\,.\end{gather}
Let $g_0\in\mathcal H$ such that $\alpha_1=\ip{g_0}{a_1 g_0}\neq 0$ and $\alpha_2=\ip{g_0}{a_2 g_0}$. Thus, one has by \eqref{eq:ip-AnonN-twosum} that $\ip{f}{(\alpha_1b_1+\alpha_2b_2)f}\geq m_A$ and 
\begin{gather*}
A\eta=\frac{a_1}{\alpha_1}\eta(\alpha_1b_1+\alpha_2b_2)+\lrp{a_2-\frac{\alpha_2}{\alpha_1}a_1}\eta b_2\,,
\end{gather*}
By renaming, one can assume $b_2$ positive definite in \eqref{eq:decm-two-sums-A}, as well as $b_1$, by using the above reasoning. 

Now, consider $t,t_0\geq0$ such that $m_{b_2-t_0b_1}=0$ and $m_{b_2-tb_1}>0$, for $t\in[0,t_0)$. So, there exist a non-zero normalized elements $\{f_n\}\subset \mathcal H$ such that $\ip{f_j}{b_2-t_0b_1 f_j}\to 0$. Thereby, 
$A\eta=(a_1+t_0a_2)\eta b_1+a_2\eta(b_2-t_0b_1)$ and
\begin{align*}
m_A\leq \lim_j  \ip{\vk{g}\vb{f_j}}{A \vk{g}\vb {f_j}}_2=\lim_j \ip{g}{(a_1+t_0a_2) g}\ip{f_j}{b_1 f_j}\,,
\end{align*}
whence $a_1+t_0a_2$ is positive definite, since $0<\ip{f_j}{b_1 f_j}\leq\no b$. Moreover, one can choose $t=t_0-\epsilon$, with $\epsilon$ small enough, to guarantee that $b_2- tb_1$ and $a_1+ta_2$ are positive definite and 
\begin{gather*}
A\eta=(a_1+ta_2)\eta b_1+a_2\eta(b_2-tb_1)\,,
\end{gather*}
Again by renaming, we may suppose that $\hat a_2,\hat b_1,\hat b_2$ in \eqref{eq:AnN-twoS} are positive definite. Finally, we can reproduce the above reasoning to guarantee (by renaming) that $\hat a_1\geq0$ and $\hat a_2,\hat b_1,\hat b_2$ are positive definite. As required.\end{proof}

Based on the last two theorems, one might expect that the operators $a_1,\ldots,b_1,\ldots$ in \eqref{eq:selfadjoint-decomposition} could also be considered positive definite, when $A$ is positive definite. Nevertheless, this is not always true due to the following.

\begin{cex}[{\cite[Sec. 4.2]{MR3119877}}]\label{cex:four-sums}
Under the condition $\dim \mathcal H=2$, it follows that  $S_2=\Span\{\varepsilon_{nm}\}_{n,m=1}^2$. Consider $A\in\mathcal B(S_2)$, given by
\begin{align*}
A\eta=(\eta_{11}+(1-t)\eta_{22})\varepsilon_{11}+t\eta_{12}\varepsilon_{12}+t\eta_{21}\varepsilon_{21}+
(\eta_{22}+(1-t)\eta_{11})\varepsilon_{22}\,,
\end{align*}
where $\eta=\sum_{n,m=1}^2\eta_{nm}\varepsilon_{nm}\in S_2$, with $\eta_{nm}\in\C$, and $t\in(0,1/2)$. Thus, the operator $A$ is positive definite since it readily follows that 
\begin{gather*}
\ip{\eta}{A\eta}_2=t\lrp{\abs{\eta_{11}}^2+\abs{\eta_{12}}^2+\abs{\eta_{21}}^2+\abs{\eta_{22}}^2}+(1-t)\abs{\eta_{11}+\eta_{22}}^2\,.
\end{gather*}
However, one can follow the same lines of the proof of \cite[Ex. 4.4]{MR3119877} to show that there is not exist $m\in\N$, non-negative operators $a_j\in\mathcal{B(H)}$ and positive definite operators $b_j\in\mathcal{B(H)}$, for $j=1,\dots,m$, for which $A\eta=\sum_{n=1}^ma_n\eta b_{n}$.
\end{cex}

As shown in the previous example, Theorems~\ref{th:one-sum-Ap} and~\ref{th:two-sum-Ap} cannot be extended to four or more summands (cf. \cite{MR3119877}). However, it is possible to provide conditions on the operators $a_1,\ldots,b_1,\ldots$ in the decomposition \eqref{eq:selfadjoint-decomposition}.

\begin{prop}\label{prop:positive-decompositionA}
For a non-negative operator $A$ in $\mathcal B(S_2)$ there exists a set of selfadjoint operators $\{a_{nm}\}\subset \mathcal{B(H)}$ such that    
\begin{gather}\label{eq:non-negative-decomposition}
A\eta=\sum_{n,m}\hat\varepsilon_{nm}\eta a_{nm}\,,\qquad \eta\in S_2\,,
\end{gather}
where $a_{11},\ldots\geq0$. Besides $a_{11},\ldots$ are positive or positive definite if $A$ is positive or positive definite, respectively, with g.l.b. $m_{a_{nn}}\geq m_A$, for all $n\in\N$.
\end{prop}
\begin{proof}
If $A$ is non-negative then so is selfadjoint. Thus, bearing in mind Lemma \ref{lem:1std-ecomposition-A} and the decomposition \eqref{eq:selfadjoint-decompo-byAmBm}, one computes that 
\begin{align}\label{eq:A-dec-PDO}
A\eta=\sum_n \varepsilon_{nn}\eta a_{nn}+\sum_{n,m\atop n\neq m} \hat\varepsilon_{nm}\eta \hat a_{nm}\,,
\end{align}
since $\hat\varepsilon_{nn}=\varepsilon_{nn}$ and $\hat a_{nn}=a_{nn}$, being  $\hat a_{nm}=\frac{1-i}{2}a_{nm}+\frac{1+i}{2}a_{mn}\in \mathcal {B(H)}$ selfadjoint. Besides, taking  the canonical basis $\{e_n\}$ of $\mathcal H$, a simple computation shows that $\ip{e_j}{\hat\varepsilon_{nm} e_j}=0$, if $n\neq m$. Thus, for $f\in\mathcal H$ and $n,m,j\in\N$,
\begin{gather}\label{eq:ip2-nneqm}
\ip{\vk{e_j}\vb f}{\hat\varepsilon_{nm}\vk{e_j}\vb f \hat a_{nm}}_2=\ip{e_j}{\hat\varepsilon_{nm} e_j}\ip{f}{\hat a_{nm} f}=0\,,\quad n\neq m\,.
\end{gather}
Thereby, one obtains by \eqref{eq:A-dec-PDO} and \eqref{eq:ip2-nneqm} that
\begin{gather}\label{eq:pd-condition}
\ip{\vk{e_j}\vb f}{A\vk{e_j}\vb f}_2=\sum_n \ip{e_j}{\varepsilon_{nn} e_j}\ip{f}{a_{nn} f}=\ip{f}{a_{jj}f}\,,
\end{gather}
whence it follows that $a_{jj}$ is non-negative. Equation \eqref{eq:pd-condition} also implies that $a_{jj}$ is positive or positive definite if $A$ is positive or positive definite, respectively, with g.l.b. $m_{a_{jj}}\geq m_A$. Hence, \eqref{eq:A-dec-PDO} yields \eqref{eq:non-negative-decomposition}.
\end{proof}

Let us give the following assertion which will be useful in the sequel.

\begin{lem}\label{lem:1st-aproox-PD}
For a positive definite operator $A\in \mathcal B(S_2)$ there exist selfadjoint operators $b_{11},a_{n,m}\in\mathcal {B(H)}$, with $n,m\in\N$, such that $b_{11},a_{11}$ and $a_{22}$ are positive definite and 
\begin{gather}\label{eq:2nd-appox}
A\eta=b_{11}\eta a_{11}+ \sum_{n,m\atop (n,m)\neq (1,1)} \hat\varepsilon_{nm}\eta a_{nm}\,.
\end{gather}
\end{lem}
\begin{proof}
One has from Proposition~\ref{prop:positive-decompositionA} that $A\eta=\sum_{n,m} \hat\varepsilon_{nm}\eta a_{nm}$, where $b_n=a_{nn}$ is positive definite with g.l.b. $m_{b_n}\geq m_A$, and $a_{nm}$ is selfadjoint,  for all $n,m\in\N$. Take $t,t_0\geq0$ for which $m_{b_2-t_0b_1}=0$ and $m_{b_2-tb_1}>0$ for $t\in[0,t_0)$. Thereby, there exist non-zero normalized elements $f_1,\ldots\in \mathcal H$ such that $\ip{f_j}{b_2-t_0b_1 f_j}\to 0$. Thus, one readily computes that 
\begin{gather}\label{eq:1st-appox}
A\eta=(\varepsilon_{11}+t_0\varepsilon_{22})\eta b_1+\varepsilon_{22}\eta(b_2-t_0b_1)+\sum_{n,m\atop (n,m)\neq (1,1),(2,2)} \hat\varepsilon_{nm}\eta a_{nm}\,.
\end{gather}
So, for $g\in\mathcal H$, by Cauchy–Schwarz inequality $ \ip{\vk{g}\vb{f_j}}{A \vk{g}\vb {f_j}}_2\leq \no A\no g^2$ and by dominated convergence theorem, one obtains from \eqref{eq:1st-appox} that 
\begin{align}\nonumber
0&<m_A\leq \lim_j  \ip{\vk{g}\vb{f_j}}{A \vk{g}\vb {f_j}}_2\\\nonumber
&=\lim_j \ip{g}{(\varepsilon_{11}+t_0\varepsilon_{22}) g}\ip{f_j}{b_1 f_j}+ \sum_{n,m\atop (n,m)\neq (1,1),(2,2)}\ip{g}{ \hat\varepsilon_{nm} g}\ip{f_j}{a_{nm} f_j}\\\label{eq:positi-defi-aux-mA}
&=\ip{g}{\lrp{\gamma_{11}(\varepsilon_{11}+t_0\varepsilon_{22})+\sum_{n,m\atop (n,m)\neq (1,1),(2,2)}\gamma_{nm}{\hat\varepsilon_{nm}}}g}\,,
\end{align}
being $\gamma_{n,m}=\lim_{j}\ip{f_j}{a_{nm} f_j}$, for $n,m\in\N$, with $(n,m)\neq(2,2)$, whence it follows that $\abs{\gamma_{n,m}}\leq\no{a_{nm}}$, and 
 $\gamma_{nn}\geq m_{b_n}>0$, for $n=1,3,\ldots$. Thus, one has by \eqref{eq:positi-defi-aux-mA} that $\gamma_{11}(\varepsilon_{11}+t_0\varepsilon_{22})+\sum_{n,m\atop (n,m)\neq (1,1),(2,2)}\gamma_{nm}{\hat\varepsilon_{nm}}$ is positive definite operator in $\mathcal{B(H)}$. Besides, we can choose $t=t_0-\epsilon$, with $\epsilon$ small enough, to guarantee that $b_2- tb_1$ and $\tilde\varepsilon_{11}=\gamma_{11}(\varepsilon_{11}+t\varepsilon_{22})+\sum_{n,m\atop (n,m)\neq (1,1),(2,2)}\gamma_{nm}{\hat\varepsilon_{nm}}$ are positive definite. Now we can use this and \eqref{eq:1st-appox} to write $A$ as
 \begin{gather*}
 A\eta=\tilde\varepsilon_{11}\eta \frac{b_1}{\gamma_{11}}+ \varepsilon_{22}\eta(b_2-tb_1)+ \sum_{n,m\atop (n,m)\neq (1,1),(2,2)} \hat\varepsilon_{nm}\eta \lrp{a_{nm}-\frac{\gamma_{nm}}{\gamma_{11}}b_1}\,,
 \end{gather*}
 which by renaming one arrives at \eqref{eq:2nd-appox}. 
\end{proof}

According to Remark~\ref{rm:selfadjoint-Asuff}, if $A\in\mathcal B(S_2)$ satisfies $A\eta=\sum_na_n\eta b_n$, where $\{a_n\},\{b_n\}\subset \mathcal{B(H)}$ meet the conditions of Theorem~\ref{th:A-nonnegative}.\eqref{it4:sA}, then $A$ is positive definite. Under certain conditions, the following result establishes the converse statement.

\begin{thm}\label{th:charat-PDO-decomp}
If $A\in  \mathcal B(S_2)$ is positive definite then there exist non-negative operators $a_3,\ldots\in\mathcal{B(H)}$ and positive definite operators $a_1,a_2,b_1,\ldots\in\mathcal{B(H)}$ such that 
\begin{gather}\label{eq:positive-definite-decomposition}
A\eta=-a_1\eta b_1+\sum_{n\geq2}a_n\eta b_{n}\,,\qquad \eta\in S_2\,.
\end{gather}
Moreover, if there exist $\zeta_n>0$, for $n=2,\dots$, such that  
\begin{gather}\label{eq:positive-definite-Cond}
b_n-\zeta_n b_1\,\mbox{ are positive definite}\quad\mbox{and}\quad 
-a_1+\sum_{n\geq2}\zeta_na_n\geq0\,,
\end{gather}
then there exist non-negative operators $\hat a_1,\ldots\in\mathcal{B(H)}$, with $\cap_n\ker\hat  a_n=\{0\}$, and positive definite operators $ \hat b_1,\ldots\in\mathcal{B(H)}$ such that 
\begin{gather}\label{eq:positive-definite-decomposition-t}
A\eta=\sum_{n}\hat  a_n\eta\hat  b_{n}\,,\qquad \eta\in S_2\,.
\end{gather}
\end{thm}
\begin{proof}
If $A$ is positive definite, then it satisfies Lemm~\ref{lem:1st-aproox-PD}. So, taking into account the decomposition \eqref{eq:2nd-appox}, since $b_{11},a_{11}, a_{22}$ are positive definite 
one can choose $\beta_{nm}>0$ sufficiently large such that $\beta_{nm}a_{22}+a_{nm}$ is positive definite, with $n,m\in\N$, $(n,m)\neq (1,1),(2,2)$. Thus, bearing in mind the selfadjoint operator $b_{22}=\varepsilon_{22}-\sum_{n,m\atop (n,m)\neq (1,1),(2,2)} \beta_{nm}\hat\varepsilon_{nm}$, one computes that
\small
\begin{align*}
A\eta&= b_{11}\eta a_{11}+ b_{22}\eta a_{22}+
\sum_{n,m\atop (n,m)\neq (1,1),(2,2)} \hat\varepsilon_{nm}\eta \lrp{\beta_{nm}a_{22}+a_{nm}}\\
&=b_{11}\eta (a_{11}+\alpha a_{22})- (\alpha b_{11}-b_{22})\eta a_{22}+
\sum_{n,m\atop (n,m)\neq (1,1),(2,2)} \hat\varepsilon_{nm}\eta \lrp{\beta_{nm}a_{22}+a_{nm}}\,,
\end{align*}\normalsize
where $\alpha>0$ is sufficiently large such that $\alpha b_{11}-b_{22}$ is positive definite. By renaming, we may suppose that  
\begin{gather}\label{eq:3rd-deco-aux-A}
A\eta=-c_{11}\eta a_{11}+c_{22}\eta a_{22}+ \sum_{n,m\atop (n,m)\neq (1,1),(2,2)}\hat\varepsilon_{nm}\eta a_{nm}\,,
\end{gather}
where $c_{11},c_{22}$ and $a_{nm}$, with $n,m\in\N$, are positive definite and the rest of the operators selfadjoint. Now taking $\lambda_{nm}>0$ large enough for which $\hat\varepsilon_{nm}+\lambda_{nm}c_{11}$ are non-negative, for $n,m\in\N$, $n\neq m$. Hence, one readily computes from by \eqref{it1-saB} and \eqref{eq:3rd-deco-aux-A} that
\begin{align*}
A\eta=\,&-c_{11}\eta \lrp{a_{11}+\sum_{n,m\atop n\neq m}\lambda_{nm}a_{nm}}+c_{22}\eta a_{22}
\\&+\sum_{n,m\atop n\neq m}\lrp{\hat\varepsilon_{nm}+\lambda_{nm}c_{11}}\eta a_{nm}+\sum_{n\geq3}\varepsilon_{nn}\eta a_{nn}\,,
\end{align*}
which again by renaming one arrives at \eqref{eq:positive-definite-decomposition}. Besides, if \eqref{eq:positive-definite-Cond} holds then one simply computes by \eqref{eq:positive-definite-decomposition} that 
\begin{gather}\label{eq:positive-definite-nonnegativeO}
A\eta=\lrp{-a_1+\sum_{n\geq2}\zeta_na_n}\eta b_1+\sum_{n\geq2}a_n\eta \lrp{b_{n}-\zeta_n b_1}\,,
\end{gather}
where $\ker{(-a_1+\sum_{n\geq2}\zeta_na_n)}\cap\cap_{n\geq2}a_n=\{0\}$, since $a_2$ is positive definite, viz. invertible. Therefore, \eqref{eq:positive-definite-nonnegativeO} yields \eqref{eq:positive-definite-decomposition-t}.
\end{proof}
\begin{rem}
It is worth pointing out that Theorem~\ref{th:charat-PDO-decomp} is also valid if we chance the hypothesis of $\{a_n\}$ and $\{b_n\}$.
\end{rem}

\subsection{A rigorous characterization of continuous sesquilinear forms and inner products on $S_2$}\label{ss-ARCIP}

We are now prepared to rigorously characterize the continuous sesquilinear forms on $S_2$, particularly the inner products on $S_2$. 

\begin{thm}
Every continuous sesquilinear  form $\varphi\colon S_2\times S_2\to\C$ admits the representation 
\begin{gather*}
\varphi(\eta,\tau)=\sum_n\tr{\eta^*a_n\tau b_n}\,, \qquad \{a_n\},\{b_n\}\subset \mathcal{B(H)}\,,
\end{gather*}
where the linear operator $\tau\mapsto \sum_na_n\tau b_n$ belongs to $\mathcal B(S_2)$. Besides, $\varphi$ is hermitian if and only if $a_1\ldots,b_1\ldots$ can be chosen as selfadjoint operators.
\end{thm}
\begin{proof}
The first part of the assertion follows from Lemma~\ref{th:cont-sesq-forms} and Theorem~\ref{th2:general-decomosition-A}, while the second part is a consequence of Lemma~\ref{th:cont-sesq-forms}.\eqref{it1:inn-p} and theorems~\ref{th:A-nonnegative}.\eqref{it1:sA}, ~\ref{th:selfadjoint-decomposition}.
\end{proof}
In the following, we present necessary and sufficient conditions for which a continuous sesquilinear form defines an inner product on $S_2$.

\begin{thm}\label{th:producing-IP}
In $\mathcal {B(H)}$, let $a_1,\ldots\geq0$,with $\cap_n \ker a_n=\{0\}$, and $b_1,\ldots>0$, such that $\eta\mapsto \sum_na_n\eta b_n$ is defined in the whole space $S_2$. Then, 
\begin{gather*}
\varphi(\eta,\tau)=\sum_n\tr{\eta^*a_n\tau b_n}\,,\quad \eta,\tau\in S_2\,,
\end{gather*}
defines an inner product on $S_2$. Besides, $\varphi$ is a definite inner product if $b_1,\ldots$ are positive definite, with $\inf_n\{m_{b_n}\}>0$.
\end{thm}
\begin{proof}
The proof is a direct consequence of Lemma~\eqref{th:cont-sesq-forms}.\eqref{it2:inn-p}, Theorem~\ref{th:A-nonnegative}.\eqref{it3:sA}-\eqref{it4:sA} and Remark~\ref{rm:selfadjoint-Asuff}.
\end{proof}

Under certain conditions, the following describes inner products.

\begin{thm}\label{th:one-two-sumas-IP}
Let $\varphi\colon S_2\times S_2\to \C$ be an inner product on $S_2$ such that $\varphi(\eta,\tau)=\sum_{n=1}^m\tr{\eta^*  a_n\tau  b_{n}}$, with $\{a_n\}_{n=1}^m,\{b_n\}_{n=1}^m\in\mathcal{B(H)}$, for $m=1,2$.
\begin{enumerate}[(i)]
\item If $m=1$ then there exist positive operators $\hat a,\hat b\in\mathcal{B(H)}$ such that
\begin{gather*}
\varphi(\eta,\tau)=\tr{\eta^*\hat a\tau\hat b}\,.
\end{gather*}
\item If $m=2$ and $\varphi$ is a definite inner product, then there exist $\hat a_1,\hat a_2\in\mathcal{B(H)}$ non-negative, whit $\ker a_1\cap\ker a_2=\{0\}$, and $b_1,b_2\in \mathcal{B(H)}$ positive definite, such that 
\begin{gather*}
\varphi(\eta,\tau)=\tr{\eta^*\hat a_1\tau\hat b_1}+\tr{\eta^*\hat a_2\tau\hat b_2}\,.
\end{gather*}
\end{enumerate}
\end{thm}
\begin{proof}
It is simple to check from Lemma~\ref{th:cont-sesq-forms} and Theorems~\ref{th:one-sum-Ap},~\ref{th:two-sum-Ap}.
\end{proof}

We conclude with the following theorem, which exhibits a general characterization of definite inner products on $S_2$. 

\begin{thm}\label{th:describing-IP} If $\varphi\colon S_2\times S_2\to \C$ defines a definite inner product on $S_2$, then there exist non-negative operators $a_3,\ldots\in\mathcal{B(H)}$ and positive definite operators $a_1,a_2,b_1,\ldots\in\mathcal{B(H)}$ such that 
\begin{gather*}
\varphi(\eta,\tau)=-\tr{\eta^* a_1\tau b_1}+\sum_{n\geq2}\tr{\eta^* a_n\tau b_n}\,,\qquad \eta,\tau\in S_2\,.
\end{gather*}
Additionally, if there exist $\zeta_n>0$, for $n=2,\dots$, such that  
\begin{gather*}
b_n-\zeta_n b_1\,\mbox{ are positive definite}\quad\mbox{and}\quad 
-a_1+\sum_{n\geq2}\zeta_na_n\,\mbox{ is non-negative}\,,
\end{gather*}
then there exist non-negative operators $\hat a_1,\ldots\in\mathcal{B(H)}$, with $\cap_n\ker\hat  a_n=\{0\}$, and positive definite operators $ \hat b_1,\ldots\in\mathcal{B(H)}$ such that 
\begin{gather*}
\varphi(\eta,\tau)=\sum_{n}\tr{\eta^*\hat  a_n\tau\hat  b_{n}}\,,\qquad \eta,\tau\in S_2\,.\end{gather*}
\end{thm}
\begin{proof}
It readily follows from Lemma~\eqref{th:cont-sesq-forms}.\eqref{it2:inn-p} and Theorem~\ref{th:charat-PDO-decomp}.
\end{proof}

\begin{rem}
Theorems~\ref{th:producing-IP}-\ref{th:describing-IP} remain true if the conditions of $\{a_n\}$ and $\{b_n\}$ are interchanged.
\end{rem}

% ------------------------------------------------------------------------

\subsection*{Acknowledgment} This work was partially supported by CONACYT-Mexico Grants CBF2023-2024-1842, CF-2019-684340, and UAM-DAI 2024: ``Enfoque Anal\'itico-Combinatorio y su Equivalencia de Estados Gaussianos".

%\subsection*{Data availability} All data generated or analyzed during this study are included in this article.
\def\cprime{$'$} \def\lfhook#1{\setbox0=\hbox{#1}{\ooalign{\hidewidth
  \lower1.5ex\hbox{'}\hidewidth\crcr\unhbox0}}} \def\cprime{$'$}
  \def\cprime{$'$} \def\cprime{$'$} \def\cprime{$'$} \def\cprime{$'$}
  \def\cprime{$'$} \def\cprime{$'$}
\providecommand{\bysame}{\leavevmode\hbox to3em{\hrulefill}\thinspace}
\providecommand{\MR}{\relax\ifhmode\unskip\space\fi MR }
\providecommand{\MRhref}[2]{%
  \href{http://www.ams.org/mathscinet-getitem?mr=#1}{#2}
}
\providecommand{\href}[2]{#2}

% ------------------------------------------------------------------------
\end{document}